\documentclass[12pt,a4paper]{amsart}
\usepackage{amsmath, amssymb, amsfonts,enumerate}

\usepackage{latexsym,amsbsy,amsmath,amscd,amssymb,enumerate}  

\usepackage{amsthm}       
\usepackage{amssymb,amsmath,latexsym}
 \newtheorem{definition}{Definition} [section]

 \newtheorem{proposition}[definition]{Proposition}       
 \newtheorem{theorem}[definition]{Theorem}       
 \newtheorem{corollary}[definition]{Corollary}       
  \newtheorem{lemma}[definition]{Lemma}

\numberwithin{equation}{section}

\newcommand{\tr}{\rm tr}
\newcommand{\Res}{\mbox{\rm Res}}

\newcommand{\Hom}{\mbox{\rm Hom}}

\def\NN{{\mathbb{N}}}

\def\FF{{\mathbb{F}}}%{{\cal F}}

\def\CC{{\mathbb{C}}}

\def\PP{{{\mathcal P}}}

\renewcommand{\dim}{\mbox{\rm dim}}

\newcommand{\CL}{{\mathbb{CL}}}

\begin{document}
\title[Mackey's criterion and Clifford groups]{Mackey's criterion for subgroup restriction of Kronecker products and harmonic analysis on Clifford groups}
\author{Tullio Ceccherini-Silberstein}
\address{Dipartimento di Ingegneria, Universit\`a del Sannio, C.so
Garibaldi 107, 82100 Benevento, Italy}
\email{tceccher@mat.uniroma1.it}
\author{Fabio Scarabotti}
\address{Dipartimento SBAI, Sapienza Universit\`a  di Roma, via A. Scarpa 8, 00161 Roma, Italy}
\email{fabio.scarabotti@sbai.uniroma1.it}
\author{Filippo Tolli}
\address{Dipartimento di Matematica e Fisica, Universit\`a Roma TRE, L. San Leonardo Murialdo 1, 00146  Roma, Italy}
\email{tolli@mat.uniroma3.it}
\date{\today}

\maketitle 
\begin{abstract}
We present a criterion for multiplicity-freeness of the decomposition of the restriction $\Res^G_H(\rho_1 \otimes \rho_2)$ of the Kronecker product of two generic irreducible representations $\rho_1, \rho_2$ of a finite group $G$ with respect to a subgroup $H \leq G$. This constitutes a generalization of a well known criterion due to Mackey (which corresponds to the case $H = G$). 
The corresponding harmonic analysis is illustated by detailed computations on the Clifford groups
$G=\CL(n)$, together with the subgroups $H=\CL(n-1)$, for $n \geq 1$, which lead to an explicit decomposition of the restriction of Kronecker products.
\end{abstract}

\tableofcontents

%%%%%%%%%%%%%%%%%%%%%%%%%%%%%%%%%%%%%%%%%%%%%%%%%%%%%
%%%%%%%%%%%%%%%%%%%%%%%%%%%%%%%%%%%%%%%%%%%%%%%%%%%%%
\section{Introduction}
%%%%%%%%%%%%%%%%%%%%%%%%%%%%%%%%%%%%%%%%%%%%%%%%%%%%%
Let $G$ be a finite group and let $K \leq G$ be a subgroup.

We denote by $\widehat{G}$ a complete set of pairwise inequivalent irreducible representations of $G$
and say that a $G$-representation is multiplicity-free provided it decomposes as the direct sum of
distinct elements in $\widehat{G}$. 
Let also denote by $L(G)$ the group algebra of $G$ and by $^K\!L(G)^K =\{f \in L(G): f(k_1gk_2) = f(g) \forall g \in G, \forall k_1,k_2 \in K\}$ the subalgebra of {bi}-$K$-invariant functions on $G$.
Recall that $(G,K)$ is a Gelfand pair provided $^K\!L(G)^K$ is commutative.

Our main result is the following criterion for  multiplicity-freeness of the restriction to $H$ of Kronecker products of irreducible representations of $G$ (this is a  generalization of Mackey's criterion corresponding to the case $H = G$).

\begin{theorem}
\label{c:mc1}
Let $G$ be a finite group and $H \leq G$ a subgroup. Consider the subgroup $\widetilde{H}=\{(h,h,h):h\in H\}$ of
$G\times G\times H$. Then the following conditions are equivalent:
\begin{enumerate}[{\rm (i)}]
\item $(G\times G\times H,\widetilde{H})$ is a Gelfand pair;
\item $\Res^G_H(\rho_1\otimes \rho_2)$ is multiplicity-free for all $\rho_1,\rho_2\in \widehat{G}$.
\end{enumerate}
\end{theorem}

We shall deduce this result from a suitable Frobenius' reciprocity type theorem for permutation representations (Theorem \ref{mainthm}) which is quite interesting on its own, and illustate it by explicit computations on the Clifford groups $\CL(n)$ (see Section \ref{s:clifford}) yielding the following results.

\begin{theorem} 
\label{t:cln-wsr}
\begin{enumerate}[{\rm (i)}]
\item $(\CL(n)\times \CL(n)\times \CL(n), \widetilde{\CL(n)})$ is a Gelfand pair;
\item $(\CL(n)\times \CL(n)\times \CL(n-1), \widetilde{\CL(n-1)})$ is a Gelfand pair if and only if $n$ is odd.
\end{enumerate}
\end{theorem}

In addition, we give a complete harmonic analysis for the Gelfand pair $(\CL(n)\times \CL(n)\times \CL(n), \widetilde{\CL(n)})$ with an explicit description of the associated spherical characters (see Section \ref{ss:clifford-sc})
and of the orbits of $\CL(n)$ acting by conjugation on the Cartesian product $\CL(n) \times \CL(n)$ (see Section \ref{ss:clifford-orbits}).

%%%%%%%%%%%%%%%%%%%%%%%%%%%%%%%%%%%%%%%%%%%%%%%%%%%%%
\section{Preliminaries}\label{Secprel}
%%%%%%%%%%%%%%%%%%%%%%%%%%%%%%%%%%%%%%%%%%%%%%%%%%%%%
%%%%%%%%%%%%%%%%%%%%%%%%%%%%%%%%%%%%%%%%%%%%%%%%%%%%%

In this section, in order to fix notation, we recall some basic facts on linear algebra and representation theory of finite groups.

All vector spaces considered here are complex. Moreover, we shall equip every finite dimensional vector space $V$ with a scalar product denoted by $\langle \cdot, \cdot\rangle_V$ with associated norm $\lVert\cdot \rVert_V$; we usually  omit the subscript if the vector space we are referring to is clear from the context. 
Given two finite dimensional vector spaces $W$ and $U$, we denote by $\Hom(W,U)$ the vector space of all linear maps  from $W$ to $U$, and for $T \in \Hom(W,U)$ we denote by $T^*\in \Hom(U,W)$ the adjoint of $T$. We define a (normalized Hilbert-Schmidt) scalar product on  $\Hom(W,U)$ by setting 
\[
\langle T_1, T_2 \rangle_{\Hom(W,U)} = \frac{1}{\dim W}\tr(T^*_2T_1)
\]
for all $T_1, T_2 \in \Hom(W,U)$, where $\tr(\cdot)$ denotes the trace of linear operators; note that, by
centrality of the trace (so that $\tr(T_2^*T_1) = \tr(T_1T_2^*)$), we have
\begin{equation}\label{PS}
\langle T_1, T_2 \rangle_{\Hom(W,U)}=\frac{\dim U}{\dim W}\langle T_2^*, T_1^* \rangle_{\Hom(U,W)} 
\end{equation}
In particular, the map $T \mapsto\sqrt{\dim U/\dim W}T^*$ is an isometry from $\Hom(W,U)$ onto $\Hom(U,W)$.
Finally, note that denoting by $I_W \colon W\rightarrow W$ the identity operator, we have $\lVert I_W\rVert_{\Hom(W,W)}=1$.\\

Let $G$ be a finite group. A \emph{unitary representation} of $G$ is a pair $(\sigma,W)$ where $W$ is a finite dimensional vector space and $\sigma \colon G \to \Hom(W,W)$ is a group homomorphism such that $\sigma(g)$ is unitary (that is, $\sigma(g)^*\sigma(g) = I_W$) for all $g \in G$. The term ``unitary'' will be omitted. We denote by $d_\sigma = \dim(W)$ the dimension of the representation $(\sigma,W)$.
Let $(\sigma, W)$ be a $G$-representation. A subspace $V \leq W$ is said to be \emph{G-invariant} provided
$\sigma(g)V \subseteq V$ for all $g \in G$. Writing $\sigma\vert_V(g) = \sigma(g)\vert_V$ for all $g \in G$, we
have that $(\sigma\vert_V,V)$ is a $G$-representation, called a \emph{subrepresentation} of $\sigma$.
One says that $\sigma$  is irreducible provided the only $G$-invariant subspaces are trivial (equivalently,
$\sigma$ admits no proper subrepresentations).

Let $(\sigma,W)$ and $(\rho, U)$ be two $G$-representations. 
We denote by $\Hom_G(W,U) = \{T \in \Hom(W,U):  T\sigma(g) = \rho(g)T, \forall g \in G\}$, the space of all {\it intertwining operators}. Observe that if $T\in\Hom_G(W,U)$ then $T^* \in\Hom_G(U,W)$. Indeed, for all $g \in G$ we have
\begin{equation}\label{adjoint}
T^*\rho(g) = T^*\rho(g^{-1})^* =(\rho(g^{-1})T)^* = (T\sigma(g^{-1}))^* = \sigma(g^{-1})^*T^* = \sigma(g)T.
\end{equation}
One says that $(\sigma,W)$ and $(\rho, U)$ are \emph{equivalent}, and we shall write $(\sigma,W) \sim (\rho, U)$
(or simply $\sigma \sim \rho$), if there exists a bijective intertwining operator $T \in \Hom_G(W,U)$.

We denote by $\widehat{G}$ a complete set of pairwise-inequivalent irreducible representations of $G$ (it is well
known (cf. \cite[Theorem 3.9.10]{book}) that there is a bijection between $\widehat{G}$ and the set of conjugacy classes of elements in $G$ so that, in particular, $\widehat{G}$ is finite). Moreover, if $\sigma, \rho\in \widehat{G}$ we set $\delta_{\sigma,\rho}=1$ (resp. $=0$) if $\sigma\sim\rho$ (resp. otherwise).

Suppose that $H \leq G$ is a subgroup. We denote by $(\Res^G_H\sigma,W)$ the
\emph{restriction} of $\sigma$ to $H$, that is, the $H$-representation defined by $[\Res^G_H\sigma](h) =
\sigma(h)$ for all $h \in H$.

The \emph{direct sum} of the two representations $\sigma$ and $\rho$ is the representation 
$(\sigma \oplus \rho, W \oplus U)$ defined by $[(\sigma \oplus \rho)(g)](w,u) = (\sigma(g)w,\rho(g)u)$
for all $g \in G$, $w \in W$ and $u \in U$. 

If $\sigma$ is a subrepresentation of $\rho$, then denoting by $W^\perp = \{u \in U: \langle u,w\rangle_U=0 \ \forall w \in W\}$ the orthogonal complement of $W$ in $U$, we have that $W^\perp$ is a $G$-invariant subspace and
$\rho = \sigma \oplus \rho\vert_{W^\perp}$. From this one deduces that every representation $\rho$ decomposes as a
direct sum of irreducible subrepresentations. In the particular case when these irreducible subrepresentations
are pairwise inequivalent, we say that $\rho$ is \emph{multiplicity-free}.

More generally, when $\sigma$ is equivalent to a subrepresentation of $\rho$, we say that $\sigma$ is \emph{contained} in $\rho$. 

Suppose that $(\sigma, W)$ is irreducible. Then the number $m = \dim \Hom_G(W, U)$ denotes the \emph{multiplicity}
of $\sigma$ in $\rho$. This means that one may decompose $U = U_1 \oplus U_2 \oplus \cdots \oplus U_m \oplus U_{m+1}$ 
with $(\rho\vert_{U_i},U_i) \sim (\sigma, W)$ for all $i=1,2,\ldots,m$. The $G$-invariant subspace
$U_1 \oplus U_2 \oplus \cdots \oplus U_m \leq U$ is called the $W$-{\it isotypic component} of $U$ and is denoted
by $mW$. One also says that $\rho$ (or, equivalently, $U$) contains  $m$ copies of $\sigma$ (resp. of $W$). If this is the case, we say that  $T_1, T_2, \ldots, T_m \in \Hom_G(W,U)$ yield an {\it isometric orthogonal decomposition} of $mW$
if $T_i \in \Hom_G(W,U_i)$ and, in addition,
\begin{equation}
\label{component}
\langle T_iw_1, T_jw_2\rangle_U = \langle w_1, w_2\rangle_W\delta_{i,j}
\end{equation}
for all $w_1, w_2 \in W$ and $i,j=1,2,\ldots,m$.
This implies that the subrepresentation of $U$ isomorphic to $mW$ is equal to the \emph{orthogonal} direct sum
$T_1W\oplus T_2W \oplus \cdots \oplus T_mW = U_1 \oplus U_2 \oplus \cdots \oplus U_m$
and each operator $T_j$ is a isometry from $W$ onto $U_j$.

\begin{lemma}\label{lemma1}
Let $(\sigma,W)$ be an irreducible $G$-representation. Then the operators $T_1, T_2, \ldots, T_m$ yield an isometric orthogonal decomposition of the $W$-isotypic component of $U$ if and only if  $T_1, T_2, \ldots, T_m$  form an orthonormal basis for $\Hom_G(W,U)$. Moreover, if this is the case, we then have:
\begin{equation}\label{orthdeltaij}
T_j^*T_i  = \delta_{i,j}I_W
\end{equation}
for all $i,j=1,2,\ldots,m$.
\end{lemma}
\begin{proof}
Suppose that $T_1, T_2, \ldots, T_m$  form an orthonormal basis for $\Hom_G(W,U)$.
Taking into account \eqref{adjoint} we have $T_j^*T_i \in \Hom_G(W,W)$. By irreducibility we deduce from Schur's lemma that there exist $\lambda_{i,j} \in \CC$ such that  $T_j^*T_i = \lambda_{i,j}I_W$ for all $i,j=1,2,\ldots,m$.
By taking traces, we get $\delta_{i,j}d_\sigma = \tr(T_j^*T_i) =  \lambda_{i,j} d_\sigma \Rightarrow \lambda_{i,j} = \delta_{i,j}$ and  \eqref{component} and \eqref{orthdeltaij} follow.
The converse implication is trivial.
\end{proof}

We denote by $L(G)$ the group algebra of $G$. This is the vector space of all functions $f \colon G \to \CC$
equipped with the \emph{convolution product} $*$ defined by setting $[f_1 * f_2](g) = \sum_{h \in G} f_1(h)f_2(h^{-1}
g) = \sum_{h \in G} f_1(gh)f_2(h^{-1})$ for all $f_1,f_2 \in L(G)$ and $g \in G$. We shall endow $L(G)$ with the
scalar product $\langle \cdot, \cdot \rangle_{L(G)}$ defined by setting
\begin{equation}
\label{e:scalar-l-g}
\langle f_1, f_2 \rangle_{L(G)} = \sum_{g \in G} f_1(g)\overline{f_2(g)}
\end{equation}
for all $f_1, f_2 \in L(G)$.

Let $(\sigma,W)$ be a representation of $G$ and let $\{w_1,w_2, \ldots, w_{d_\sigma}\}$ be an orthonormal basis of $W$. The corresponding \emph{matrix coefficients} $u_{j,i}\in L(G)$  are defined by setting 
\begin{equation}\label{stellap2}
u_{j,i}(g) = \langle \sigma(g)w_i, w_j\rangle
\end{equation}
for all $i,j = 1,2, \ldots, d_\sigma$ and $g \in G$.

\begin{proposition}
Let $\sigma, \rho \in \widehat{G}$. Then 
\begin{equation}\label{ORT}
\langle u_{i,j}^\sigma, u_{h,k}^\rho\rangle= \frac{|G|}{d_\sigma} \delta_{\sigma,
\rho}\delta_{i,h}\delta_{j,k} 
\mbox{   (orthogonality relations)}
\end{equation}
and
\begin{equation}\label{CON} 
u_{i,j}^\sigma* u_{h,k}^\rho = \frac{|G|}{d_\sigma} \delta_{\sigma,
\rho}\delta_{j,h} u^\sigma_{i,k}
\mbox{    (convolution properties)}
\end{equation}
for all $i,j = 1,2, \ldots, d_\sigma$ and $h,k = 1,2, \ldots, d_\rho$.
\end{proposition}
\begin{proof}
See  Lemma 3.6.3 and Lemma 3.9.14  of \cite{book}.
\end{proof}

The sum $\chi_\sigma = \sum_{i=1}^{d_\sigma} u_{i,i}\in L(g)$ of the diagonal entries of matrix coefficients is called the \emph{character} of $\sigma$. 

Given a vector space $V$, we denote by $V'= \Hom(V,\CC)$ its \emph{dual space}.
The \emph{conjugate} of a linear operator $T \in \Hom(V,V)$ is the linear operator $T' \in \Hom(V',V')$
defined by $[T'(v')](v) = v'(T(v))$ for all $v' \in V'$ and $v \in V$.
The \emph{conjugate} of a $G$-representation $(\sigma,W)$ is the
$G$-representation $(\sigma',W')$ defined by $\sigma'(g) = \left(\sigma(g^{-1})\right)'$ for all $g \in G$.
Note that $d_\sigma = d_{\sigma'}$ and that $\sigma$ is irreducible if and only if $\sigma'$ is.
Note that one may have $\sigma \not\sim \sigma'$ but it is always the case that $(\sigma')' \sim \sigma$.

Let $G_1$ and $G_2$ be two finite groups and suppose that $(\sigma_i,W_i)$ is a $G_i$-representation
for $i=1,2$. The \emph{outer tensor product} of $\sigma_1$ and $\sigma_2$ is the $(G_1 \times G_2)$-representation
$(\sigma_1 \boxtimes \sigma_2, W_1 \otimes W_2)$ defined on simple tensors by setting
$[(\sigma_1 \boxtimes \sigma_2)(g_1,g_2)](w_1 \otimes w_2) = \left(\sigma_1(g_1)w_1\right) \otimes \left(\sigma_2(g_2)w_2\right)$ for all $g_i \in G_i$ and $w_i \in W_i$, $i=1,2$ and then extending by linearity.
When $G_1 = G_2 = G$, we denote by $\widetilde{G} = \{(g,g): g \in G\}$ the \emph{diagonal subgroup} of $G \times G$
(this is clearly isomorphic to and we shall thus identify it with $G$) we set $\sigma_1 \otimes \sigma_2 = \Res^{G \times G}_{G}(\sigma_1 \boxtimes \sigma_2)$ and call the $G$-representation $(\sigma_1 \otimes \sigma_2,W_1 \otimes W_2)$ the \emph{inner tensor product} of $\sigma_1$ and $\sigma_2$.

In general, the inner tensor product of two irreducble $G$-representations is not irreducible: 
the following easy lemma, however, gives a sufficient condition guaranteeing its irreducibility
(we include the proof for the readers' convenience). 

\begin{lemma}
\label{l:filippo}
Let $G$ be a finite group and let $\sigma_1, \sigma_2 \in \widehat{G}$. 
Suppose that $\dim \sigma_1 = 1$.
Then $\sigma_1 \otimes \sigma_2$ is irreducible.
\end{lemma}
\begin{proof}
We have
\[
\begin{split}
\langle \chi_{\sigma_1 \otimes \sigma_2}, \chi_{\sigma_1 \otimes \sigma_2} \rangle & = \langle \chi_{\sigma_1} \chi_{\sigma_2}, \chi_{\sigma_1} \chi_{\sigma_2} \rangle \\
& = \frac{1}{|G|} \sum_{g \in G} \chi_{\sigma_1}(g) \chi_{\sigma_2}(g) \overline{\chi_{\sigma_1}(g) \chi_{\sigma_2}(g)}\\
\mbox{(since $\overline{\chi_{\sigma_1}(g)} = \chi_{\sigma_1}(g^{-1})$)} \ & = \frac{1}{|G|} \sum_{g \in G} \chi_{\sigma_1}(g)\chi_{\sigma_1}(g^{-1}) \chi_{\sigma_2}(g) \overline{\chi_{\sigma_2}(g)}\\
\mbox{(since $\chi_{\sigma_1}\sim \sigma_1$ is multiplicative)} \ & = \frac{1}{|G|} \sum_{g \in G} \chi_{\sigma_1}(gg^{-1}) \chi_{\sigma_2}(g) \overline{\chi_{\sigma_2}(g)}\\
& = \frac{1}{|G|} \sum_{g \in G} \chi_{\sigma_2}(g) \overline{\chi_{\sigma_2}(g)}\\
& = 1.
\end{split}
\]
Alternatively, up to identifying $V_{\sigma_1 \otimes \sigma_2}$ with $V_{\sigma_2}$, which we simply denote by $V$,
suppose that $W \\subseteq V$ is a $\sigma_1 \otimes \sigma_2$ invariant subspace and let $g \in G$.
We then have $\sigma_1(g)\sigma_2(g)W = (\sigma_1 \otimes \sigma_2)(g)W \subseteq W$ and, by multiplying on the
left by the scalar $\overline{\sigma_1(g)}$, we obtain $\sigma_2(g)W \subseteq \overline{\sigma_1(g)}W = W$. Thus
$W \\subseteq V$ is $\sigma_2$ invariant and since $\sigma_2$ is irreducible, $W$ is trivial.
\end{proof}

Suppose that $G$ acts transitively on a finite set $X$ and denote by
$L(X)$ the vector space of all functions $f \colon X \to \CC$.
The associated \emph{permutation representation} $(\lambda,L(X))$  is defined by 
$[\lambda(g)f](x)=f(g^{-1}x)$ for all $f\in L(X)$, $x\in X$ and $g\in G$.

Fix a point $x_0 \in X$ and denote by $K = \{g \in G: gx_0 = x_0\} \leq G$ its \emph{stabilizer}. 
Then we may identify $X$ with the set $G/K$ of left cosets of $K$ in $G$ as homogeneous $G$-spaces and
$L(X)$ with the $L(G)$-subalgebra $L(G)^K = \{f \in L(G): f(gk) = f(g) \forall g \in G, \forall k \in K\}$ 
of  $K$-\emph{right-invariant} functions. Moreover we denote by $^K\!L(G)^K =\{f \in L(G): f(k_1gk_2) = f(g) \forall g \in G, \forall k_1,k_2 \in K\}$ the subalgebra of \emph{bi}-$K$-\emph{invariant} functions.

Given a representation $(\sigma,W)$ we denote by $W^K = \{w \in W: \sigma(k)w = w, \forall k \in K\} \leq W$ the subspace of $K$-\emph{invariant vectors} of $W$.

For the following result we refer to \cite{GelGeorg} and/or to the monographs \cite[Chapter 4]{book} and \cite{Diaconis}.

\begin{theorem} 
\label{t:GP}
The following conditions are equivalent:
\begin{enumerate}[{\rm (a)}]
\item The algebra $^K\!L(G)^K$ is commutative;
\item the permutation representation $(\lambda,L(X))$ is multiplicity-free;
\item the algebra $\Hom_G(L(X),L(X))$ is commutative;
\item for every $(\sigma,W) \in \widehat{G}$ one has $\dim(W^K) \leq 1$;
\item for every $(\sigma,W) \in \widehat{G}$ one has $\dim \Hom_G(W,L(X)) \leq 1$.
\end{enumerate}
\end{theorem}

\begin{definition}
\label{d:GP}{\rm 
If one of the equivalent conditions in Theorem \ref{t:GP} is satisfied, on says that $(G,K)$ is a
\emph{Gelfand pair}.}
\end{definition}

The equivalence (d) $\Leftrightarrow$ (e) in Theorem \ref{t:GP} is a particular case of the
fact that the multiplicity of $(\sigma,W)$ in $(\lambda,L(X))$ is equal to $\dim W^K$. 
More precisely (see  \cite[Section 1.2.1]{book3} or \cite{st1}), with every $w\in W^K$ we may associate the linear
map $T_w \colon W \rightarrow L(X)$ defined by setting
\[
(T_wv)(x)=\sqrt{\frac{d_\sigma}{\lvert X\rvert}}\langle v,\sigma(g)w\rangle_V
\]
for all $v\in W$, $x\in X$ (here $g$ is any element in $G$ such that $gx_0=x$). 
Then one easily checks that $T_w \in \Hom_G(W,L(X))$ and that 
\[
\langle T_{w_1}v_1, T_{w_2}w_2 \rangle_{L(X)}=\langle v_1,v_2\rangle_W\langle w_2,w_1\rangle_W
\]
for all $w_1,w_2,v_1,v_2 \in W$.
This easily implies that the map 
\begin{equation}\label{isomTv}
w\longrightarrow T_w
\end{equation}
yields an isometry between $W^K$ and $\Hom_G(W,L(X))$. 

An irreducible representation $(\sigma,W) \in \widehat{G}$ contained in $(\lambda,L(X))$
is called \emph{spherical} and its {\em spherical character} is the function $\chi_\sigma^K \in L(G)$ defined by 
\begin{equation}\label{sphchar}
\chi_\sigma^K(g)=\frac{1}{\lvert K\rvert}\sum_{k\in K}\overline{\chi_\sigma(kg)}
\end{equation}
for all $g \in G$, where $\chi_\sigma$ denotes the (usual) character of $\sigma$
(see  \cite[Section 1.2.2]{book3} or \cite{st1}). In particular,
\begin{equation}\label{sphcharmult}
\chi_\sigma^K(1_G)=\text{ multiplicity of }\sigma\text{ in }\lambda.
\end{equation}
%%%%%%%%%%%%%%%%%%%%%%%%%%%%%%%%%%%%%%%%%%%%%%%%%%%%%
%%%%%%%%%%%%%%%%%%%%%%%%%%%%%%%%%%%%%%%%%%%%%%%%%%%%%
\section{Harmonic analysis for the pair $(G \times G \times H, \tilde{H})$}
\label{s:FR}
%%%%%%%%%%%%%%%%%%%%%%%%%%%%%%%%%%%%%%%%%%%%%%%%%%%%%
%%%%%%%%%%%%%%%%%%%%%%%%%%%%%%%%%%%%%%%%%%%%%%%%%%%%%

Let $G$ be a finite group and $H \leq G$ be a subgroup.

We define a transitive action of the group $G\times G\times H$ on the set $X=G\times G$ by setting
\begin{equation}
\label{e:azione}
(g_1,g_2,h)(g_3,g_4)=(g_1g_3g_2^{-1},g_2g_4h^{-1})
\end{equation}
for all $(g_1,g_2,h)\in G\times G\times H$ and $(g_3,g_4)\in G\times G$. 
The stabilizer of $(1_G,1_G)\in G\times G$ is the subgroup $\widetilde{H}=\{(h,h,h):h\in H\}$. 
Denote by $\eta$ the associated permutation representation of $G\times G\times H$, so that
\[
[\eta(g_1,g_2,h)f](g_3,g_4)=f(g_1^{-1}g_3g_2,g_2^{-1}g_4h)
\]
for all $f\in L(G \times G)$, $(g_1,g_2,h)\in G\times G\times H$ and $(g_3,g_4)\in G\times G$. 

\subsection{A Frobenius reciprocity type theorem and proof of the main result}
\label{ss:FR}
We shall deduce Theorem \ref{c:mc1} from the following result which can be regarded as a sort of
Frobenius reciprocity for permutation representations.

\begin{theorem}
\label{mainthm}
Let $(\rho_1,V_1)$ and $(\rho_2,V_2)$ be two {\em irreducible} $G$-representations and let $(\theta,W)$ be an $H$-representation. 
Then the map
\[
\begin{array}{ccc}
\Hom_{G\times G\times H}(\rho_1\boxtimes \rho_2\boxtimes \theta,\eta)&\longrightarrow&\Hom_H\left(\Res^G_H(\rho_1\otimes\rho_2),\theta'\right)\\
T&\longmapsto&\widetilde{T}
\end{array}
\]
where
\[
\left[\widetilde{T}(v_1\otimes v_2)\right](w)=\frac{\lvert G\rvert}{\sqrt{d_\theta}}[T(v_1\otimes v_2\otimes w)](1_G,1_G)
\]
for all $v_1\in V_1$, $v_2\in V_2$ and $w\in W$,
is an isometric isomorphism of vector spaces. Its inverse is given by the map
\[
\begin{array}{ccc}
\Hom_H\left(\Res^G_H(\rho_1\otimes\rho_2\right),\theta')&\longrightarrow&\Hom_{G\times G\times H}(\rho_1\boxtimes \rho_2\boxtimes \theta,\eta)\\
S&\longmapsto&\widehat{S}
\end{array}
\]
where
\[
\left[\widehat{S}(v_1\otimes v_2\otimes w)\right](g_1,g_2)=\frac{\sqrt{d_\theta}}{\lvert G\rvert}\left\{S[\rho_1(g_2^{-1}g_1^{-1})v_1\otimes \rho_2(g_2^{-1})v_2]\right\}(w)
\]
for all $v_1\in V_1$, $v_2\in V_2$, $w\in W$ and $g_1,g_2\in G$. 
\end{theorem}
\begin{proof}
First of all, note that a linear operator $T:V_1\otimes V_2\otimes W\rightarrow L(G\times G)$ intertwines $\rho_1\boxtimes \rho_2\boxtimes \theta$ with $\eta$ if and only if 
\begin{equation}\label{Tintertw}
\left\{T[\rho_1(g_1)v_1\otimes\rho_2(g_2)v_2\otimes\theta(h)w]\right\}(g_3,g_4)=\left[T(v_1\otimes v_2\otimes w)\right](g_1^{-1}g_3g_2,g_2^{-1}g_4h),
\end{equation}
while a linear operator $S:V_1\otimes V_2\rightarrow W'$ intertwines $\Res^G_H(\rho_1\otimes \rho_2)$ with $\theta'$ if and only if
\begin{equation}\label{Sintertw}
\left\{S[\rho_1(h)v_1\otimes\rho_2(h)v_2]\right\}(w)=\left[\theta'(h)S(v_1\otimes v_2)\right](w),
\end{equation}
for all $v_1\in V_1$, $v_2\in V_2$, $w\in W$, $g_1,g_2,g_3,g_4\in G$ and $h\in H$. It follows that:
\[
\begin{split}
\left\{\widetilde{T}[\rho_1(h)v_1\otimes\rho_2(h)v_2]\right\}(w)=&\left\{T[\rho_1(h)v_1\otimes\rho_2(h)v_2\otimes w]\right\}(1_G,1_G)\\
(\text{by }\eqref{Tintertw})\qquad=&\frac{\lvert G\rvert}{\sqrt{d_\theta}}\left\{T[v_1\otimes v_2\otimes w]\right\}(1_G,h^{-1})\\
(\text{again by }\eqref{Tintertw})\qquad=&\frac{\lvert G\rvert}{\sqrt{d_\theta}}\left\{T[v_1\otimes v_2\otimes \theta(h^{-1})w]\right\}(1_G,1_G)\\
=&\left[\widetilde{T}(v_1\otimes v_2)\right]\left(\theta(h^{-1})w\right)\\
=&\left[\theta'(h)\widetilde{T}(v_1\otimes v_2)\right](w).
\end{split}
\]
That is, $\widetilde{T}$ intertwines $\Res^G_H(\rho_1\otimes \rho_2)$ with $\theta'$. Similarly,
\[
\begin{split}
\left\{\widehat{S}[\rho_1(g_1)v_1\otimes\rho_2(g_2)v_2\otimes\theta(h)w]\right\}(g_3,g_4)=&\frac{\sqrt{d_\theta}}{\lvert G\rvert}\left\{S[\rho_1(g_4^{-1}g_3^{-1}g_1)v_1\otimes \rho_2(g_4^{-1}g_2)v_2]\right\}(\theta(h)w)\\
=&\frac{\sqrt{d_\theta}}{\lvert G\rvert}\left\{\theta'(h^{-1})S[\rho_1(g_4^{-1}g_3^{-1}g_1)v_1\otimes \rho_2(g_4^{-1}g_2)v_2]\right\}(w)\\
(\text{by }\eqref{Sintertw})\qquad=&\frac{\sqrt{d_\theta}}{\lvert G\rvert}\left\{S[\rho_1(h^{-1}g_4^{-1}g_3^{-1}g_1)v_1\otimes \rho_2(h^{-1}g_4^{-1}g_2)v_2]\right\}(w)\\
=&\left[\widehat{S}(v_1\otimes v_2\otimes w)\right](g_1^{-1}g_3g_2,g_2^{-1}g_4h)
\end{split}
\]
and therefore $\widehat{S}$ intertwines $\rho_1\boxtimes \rho_2\boxtimes \theta$ with $\eta$. 
Moreover, we have
\[
\begin{split}
\left[\widehat{\widetilde{T}}(v_1\otimes v_2\otimes w)\right](g_1,g_2)=&\frac{\sqrt{d_\theta}}{\lvert G\rvert}\left\{\widetilde{T}[\rho_1(g_2^{-1}g_1^{-1})v_1\otimes \rho_2(g_2^{-1})v_2]\right\}(w)\\
=&\left\{T[\rho_1(g_2^{-1}g_1^{-1})v_1\otimes \rho_2(g_2^{-1})v_2\otimes w]\right\}(1_G,1_G)\\
(\text{by }\eqref{Tintertw})\qquad=&\left[T(v_1\otimes v_2\otimes w)\right](g_1,g_2)
\end{split}
\]
and
\[
\begin{split}
\left[\widetilde{\widehat{S}}(v_1\otimes v_2)\right](w)=&\frac{\lvert G\rvert}{\sqrt{d_\theta}}\left[\widehat{S}[(v_1\otimes v_2\otimes w)\right](1_G,1_G)\\
=&\left[S(v_1\otimes v_2)\right](w),
\end{split}
\]
that is, $\widehat{\widetilde{T}}=T$ and $\widetilde{\widehat{S}}=S$. It follows that the map $T \mapsto \widetilde{T}$ is a bijection and its inverse is $S \mapsto \widehat{S}$. 
Finally for $i=1,2$ let $\{v_{i,j}:j=1,2,\dotsc,d_{\rho_i}\}$ be an orthonormal basis in $V_i$ and denote by $u^{\rho_i}_{h,j}$ the associated matrix coefficients; let also $\{w_j:j=1,2,\dotsc,d_\theta\}$ be an orthonormal basis in $W$. If $T_1,T_2\in\Hom_{G\times G\times H}(\rho_1\boxtimes \rho_2\boxtimes \theta,\eta)$ then we have, on the one hand,
\[
\begin{split}
\left\langle\widetilde{T_1},\widetilde{T_2}\right\rangle=&\frac{1}{d_{\rho_1}d_{\rho_2}}\tr(\widetilde{T_2}^*\widetilde{T_1})\\
=&\frac{1}{d_{\rho_1}d_{\rho_2}}\sum_{i=1}^{d_{\rho_1}}\sum_{j=1}^{d_{\rho_2}}\left\langle\widetilde{T_1}(v_{1,i}\otimes v_{2,j}),\widetilde{T_2}(v_{1,i}\otimes v_{2,j})\right\rangle_{W'}\\
=&\frac{1}{d_{\rho_1}d_{\rho_2}}\sum_{i=1}^{d_{\rho_1}}\sum_{j=1}^{d_{\rho_2}}\sum_{\ell=1}^{d_\theta}\left[\widetilde{T_1}(v_{1,i}\otimes v_{2,j})\right](w_\ell)\cdot\overline{\left[\widetilde{T_2}(v_{1,i}\otimes v_{2,j})\right](w_\ell)}\\
=&\frac{\lvert G\rvert^2}{d_{\rho_1}d_{\rho_2}d_\theta}\sum_{h=1}^{d_{\rho_1}}\sum_{m=1}^{d_{\rho_2}}\sum_{\ell=1}^{d_\theta}
\left[T_1(v_{1,h}\otimes v_{2,m}\otimes w_\ell)\right](1_G,1_G)\cdot\\
&\qquad\qquad\qquad\qquad\cdot\overline{\left[T_2(v_{1,h}\otimes v_{2,m}\otimes w_\ell)\right](1_G,1_G)}
\end{split}
\]
and, on the other hand,
\[
\begin{split}
\left\langle T_1,T_2\right\rangle=&\frac{1}{d_{\rho_1}d_{\rho_2}d_\theta}\sum_{i=1}^{d_{\rho_1}}\sum_{j=1}^{d_{\rho_2}}\sum_{\ell=1}^{d_\theta}\sum_{g_1,g_2\in G}\left[T_1(v_{1,i}\otimes v_{2,j}\otimes w_\ell)\right](g_1,g_2)\cdot\\
&\qquad\qquad\qquad\qquad\qquad\qquad\cdot\overline{\left[T_2(v_{1,i}\otimes v_{2,j}\otimes w_\ell)\right](g_1,g_2)}\\
(\text{by }\eqref{Tintertw})\quad=&\frac{1}{d_{\rho_1}d_{\rho_2}d_\theta}\sum_{i=1}^{d_{\rho_1}}\sum_{j=1}^{d_{\rho_2}}\sum_{\ell=1}^{d_\theta}\sum_{g_1,g_2\in G}\left[T_1(\rho_1(g_2^{-1}g_1^{-1})v_{1,i}\otimes \rho_2(g_2^{-1})v_{2,j}\otimes w_\ell)\right](1_G,1_G)\cdot\\
&\qquad\qquad\qquad\qquad\cdot\overline{\left[T_2(\rho_1(g_2^{-1}g_1^{-1})v_{1,i}\otimes \rho_2(g_2^{-1})v_{2,j}\otimes w_\ell)\right](1_G,1_G)}\\
=&\frac{1}{d_{\rho_1}d_{\rho_2}d_\theta}\sum_{i,h,s=1}^{d_{\rho_1}}\sum_{j,m,t=1}^{d_{\rho_2}}\sum_{\ell=1}^{d_\theta}\sum_{g_2\in G}u_{m,j}^{\rho_2}(g_2^{-1})\overline{u_{t,j}^{\rho_2}(g_2^{-1})}\sum_{g_1\in G}u_{h,i}^{\rho_1}(g_2^{-1}g_1^{-1})\overline{u_{s,i}^{\rho_1}(g_2^{-1}g_1^{-1})}\cdot\\
&\cdot \left[T_1(v_{1,h}\otimes v_{2,m}\otimes w_\ell)\right](1_G,1_G)\cdot\overline{\left[T_2(v_{1,s}\otimes v_{2,t}\otimes w_\ell)\right](1_G,1_G)}\\
(\text{by }\eqref{ORT})\quad=&\frac{\lvert G\rvert^2}{d_{\rho_1}d_{\rho_2}d_\theta}\sum_{h=1}^{d_{\rho_1}}\sum_{m=1}^{d_{\rho_2}}\sum_{\ell=1}^{d_\theta}
\left[T_1(v_{1,h}\otimes v_{2,m}\otimes w_\ell)\right](1_G,1_G)\cdot\\
&\qquad\qquad\qquad\qquad\cdot\overline{\left[T_2(v_{1,h}\otimes v_{2,m}\otimes w_\ell)\right](1_G,1_G)}.
\end{split}.
\]
We deduce that $\left\langle\widetilde{T_1},\widetilde{T_2}\right\rangle_{\Hom({V_1 \otimes V_2, W'})}=\langle T_1,T_2\rangle_{\Hom(V_1 \otimes V_2 \otimes W, L(G \times G))}$ showing that the bijective map $T\mapsto\widetilde{T}$ is indeed an isometry.
\end{proof}

\begin{proof}[Proof of Theorem \ref{c:mc1}]
We first observe that if an $H$-representation $\theta$ is irreducible, then the multiplicity of the irreducible $(G \times G \times H)$-representation $\rho_1\boxtimes\rho_2\boxtimes\theta$ in the permutation representation $\eta$ (associated with the pair $(G \times G \times H,\widetilde{H})$)
equals the multiplicity of the irreducible $H$-representation $\theta'$ in $\Res^G_H(\rho_1\otimes \rho_2)$.
Moreover, all irreducible $(G \times G \times H)-$ (resp. $H$-)representations are, up to equivalence, equal to $\rho_1\boxtimes\rho_2\boxtimes\theta$ (resp. $\theta'$) with $\rho_1, \rho_2 \in \widehat{G}$ and $\theta \in \widehat{H}$. Theorefore $\eta$ is multiplicity-free (and therefore, all equivalent conditions in Theorem \ref{t:GP} are satisfied for the pair $(G \times G \times H,\widetilde{H})$) if and only if $\Res^G_H(\rho_1\otimes \rho_2)$ is multiplicity-free.
\end{proof}

\subsection{Invariant vectors and spherical characters}
In this section we describe the isometry \eqref{isomTv} and the spherical characters for the pair
$(G \times G \times H, \widetilde{H})$.

Using the description of tensor products in \cite[Section 9.1]{book} (see also \cite{Simon}),
we identify  $V_1\otimes V_2\otimes W$ with the vector space of all {\em anti-trilinear} maps $B\colon V_1\times V_2\times W \to \CC$.
Moreover, an element $B\in V_1\otimes V_2\otimes W$ is ${\widetilde{H}}$-invariant (that is,
$B\in\left(V_1\otimes V_2\otimes W\right)^{\widetilde{H}}$) if and only if
\[
B(\rho_1(h)v_1,\rho_2(h)v_2,\theta(h)w)=B(v_1,v_2,w)
\]
for all $v_1\in V_1, v_2 \in V_2, w \in W$ and $h \in H$.

\begin{proposition}\label{propTB}
Let $B\in\left(V_1\otimes V_2\otimes W\right)^{\widetilde{H}}$. 
Then the associated intertwing operator $T_B\in\Hom_{G\times G\times H}(\rho_1\boxtimes \rho_2\boxtimes \theta,\eta)$ given by \eqref{isomTv} has the form:
\[
[T_B(v_1\otimes v_2\otimes w)](g_1,g_2)=\frac{\sqrt{d_{\rho_1}d_{\rho_2}d_\theta}}{\lvert G\rvert}\overline{B(\rho_1(g_2^{-1}g_1^{-1})v_1, \rho_2(g_2^{-1})v_2, w)}.
\]
for all $v_1 \in V_1, v_2 \in V_2$ and $w \in W$.
\end{proposition}
\begin{proof}
First note that
\[
\langle v_1\otimes v_2\otimes w,B\rangle=\overline{\langle B, v_1\otimes v_2\otimes w\rangle}=\overline{B(v_1,v_2, w\rangle}
\]
and that the last expression is {\em trilinear}. Moreover, recalling \eqref{e:azione}, we have 
$(g_3,g_4,h)(1_G,1_G)=(g_1,g_2)$ if we take $g_3=g_1g_2, g_4=g_2$ and $h=1_G$. 
By virtue of these simple considerations we may write \eqref{isomTv} explicitely as follows:
\[
\begin{split}
[T_B(v_1\otimes v_2\otimes w)](g_1,g_2)=&\frac{\sqrt{d_{\rho_1}d_{\rho_2}d_\theta}}{\lvert G\rvert}\left\langle v_1\otimes v_2\otimes w,[\rho_1(g_1g_2)\boxtimes \rho_2(g_2)\boxtimes \theta(1_G)]B\right\rangle\\
=&\frac{\sqrt{d_{\rho_1}d_{\rho_2}d_\theta}}{\lvert G\rvert}\left\langle \rho_1(g_2^{-1}g_1^{-1})v_1\otimes \rho_2(g_2^{-1})v_2\otimes w,B\right\rangle\\
=&\frac{\sqrt{d_{\rho_1}d_{\rho_2}d_\theta}}{\lvert G\rvert}\overline{B(\rho_1(g_2^{-1}g_1^{-1})v_1, \rho_2(g_2^{-1})v_2, w)}.
\end{split}
\]
\end{proof}

By combining the isometries in Theorem \ref{mainthm} and Proposition \ref{propTB} we obtain the following.
\begin{corollary}
The map $B\mapsto\widetilde{T_B}$ yields an isometry 
\[
\left(V_1\otimes V_2\otimes W\right)^{\widetilde{H}} \to \Hom_H\left(\Res^G_H(\rho_1\otimes\rho_2),\theta'\right).
\] 
Moreover, for $B \in V_1\otimes V_2\otimes W$, the intertwining operator $\widetilde{T_B}$ has the form:
\[
\left[\widetilde{T_B}(v_1\otimes v_2)\right](w)=\sqrt{d_{\rho_1}d_{\rho_2}}\;\overline{B(v_1,v_2,w)}
\] 
for all $v_1 \in V_1, v_2 \in V_2$ and $w \in W$.
\end{corollary}

Clearly, the last Corollary may be proved directly following the same arguments in the proof of Theorem \ref{mainthm}. 

We end this section by providing an explicit expression of \eqref{sphchar} for the pair
$(G \times G \times H, \widetilde{H})$. This quite easy: the spherical character associated to the for the pair
$(G \times G \times H$-representation $\rho_1\boxtimes\rho_2\boxtimes\theta$ is given by:
\[
\begin{split}
\psi_{\rho_1\boxtimes\rho_2\boxtimes\theta}(g_1,g_2,h_1)=&\frac{1}{\lvert H\rvert}\sum_{h\in H}\overline{\chi_{\rho_1\boxtimes\rho_2\boxtimes\theta}(hg_1,hg_2,hh_1)}\\
=&\frac{1}{\lvert H\rvert}\sum_{h\in H}\overline{\chi_{\rho_1}(hg_1)}\:\overline{\chi_{\rho_2}(hg_2)}\;\overline{\chi_{\theta}(hh_1)}
\end{split}
\]
for all $g_1, g_2 \in G$ and $h_1 \in H$.
In particular, \eqref{sphcharmult} and the isomorphisms above yield
\[
\begin{split}
\psi_{\rho_1\boxtimes\rho_2\boxtimes\theta}(1_G,1_G,1_G)=&
\dim\left(V_1\times V_2\times W\right)^{\widetilde{H}}\\
=&\dim\Hom_{G\times G\times H}(\rho_1\boxtimes \rho_2\boxtimes \theta,\eta)\\
=&\dim\Hom_H\left(\Res^G_H(\rho_1\otimes\rho_2),\theta'\right)\\
\equiv&\text{ multiplicity of } \theta'\text{ in }\Res^G_H(\rho_1\otimes\rho_2).
\end{split}
\]
The fact that $\psi_{\rho_1\boxtimes\rho_2\boxtimes\theta}(1_G,1_G,1_G)$ equals the
multiplicity of $\theta'$ in $\Res^G_H(\rho_1\otimes\rho_2)$ could also be recovered from the explicit expression of 
$\psi_{\rho_1\boxtimes\rho_2\boxtimes\theta}(1_G,1_G,1_G)$ which is nothing but the scalar product between the character of $\theta'$ and the character of $\Res^G_H(\rho_1\otimes\rho_2)$.

%%%%%%%%%%%%%%%%%%%%%%%%%%%%%%%%%%%%%%%%%%%%%%%%%
\section{Harmonic analysis on Clifford groups}
\label{s:clifford}
%%%%%%%%%%%%%%%%

In this section we consider the Clifford groups $G=\CL(n)$ and we study the decomposition of the tensor product
of two irreducible representations and its restriction to the subgroup $H=\CL(n-1)$.

\subsection{Clifford groups and their irreducible representations}
Let $n \in \NN$ and set $X_n = \{1,2,\ldots,n\}$. Denote by $\CL(n)$ the \emph{Clifford group} of degree $n$.
Recall that $\CL(n) = \{\pm\gamma_A: A \subseteq X_n\}$ with multiplication given by
\[
\varepsilon_1 \gamma_A \cdot \varepsilon_2 \gamma_B = \varepsilon_1 \varepsilon_2 (-1)^{\xi(A,B)} \gamma_{A \triangle B}
\]
where $\triangle$ denotes the symmetric difference of two sets and  $\xi(A,B)$ equals the number of elements $(a,b) \in A \times B$ such that $a > b$, for all $\varepsilon_1, \varepsilon_2 \in \{1,-1\}$ and $A,B \subseteq X_n$. Notice that the identity element is given by $1 = \gamma_\varnothing$ and
that $(\varepsilon\gamma_A)^{-1} = \varepsilon (-1)^\frac{|A|(|A|-1)}{2}\gamma_{A}$ for all $\varepsilon = \pm 1$
and $A \subseteq X_n$.

It is well known (cf. \cite[Section IV.3]{Simon}) that $\CL(n)$ admits exactly $2^n$ one-dimensional representations,
namely $\chi_A$, $A \subseteq X_n$ given by
\begin{equation}
\label{e:chiA}
\chi_A(\pm\gamma_B) = (-1)^{|A \cap B|}
\end{equation}
for all $B \subseteq X_n$, and, if $n$ is even, there is only one irreducible representation $\rho_n$, of
dimension $2^{n/2}$ whose character is given by
\begin{equation}
\label{e:chirhon}
\chi_{\rho_n}(\pm\gamma_B) = \pm \delta_{B,\varnothing}2^{n/2}
\end{equation}
for all $B \subseteq X_n$, and, if $n$ is odd, say $n = 2m+1$, there are exactly two irreducible representation $\rho_n^\pm$, of
dimension $2^{m}$ whose characters are given by
\begin{equation}
\label{e:chirhonpm}
\chi_{\rho_n^\pm}(\gamma_B) = 
\begin{cases} 2^m & \mbox{ if } B = \varnothing\\
0 & \mbox{ if } B \neq \varnothing, X_n\\
\pm c 2^m & \mbox{ if } B = X_n
\end{cases}
\end{equation}
where
\begin{equation}
\label{e:qua}
c = \begin{cases} 1 & \mbox{ if } m \equiv 0 \mod 2\\
-i & \mbox{ if } m \equiv 1 \mod 2
\end{cases}
\end{equation}
and
\begin{equation}
\label{e:trestelle}
\chi_{\rho_n^\pm}(-\gamma_B) = -\chi_{\rho_n^\pm}(\gamma_B) 
\end{equation}
for all $B \subseteq X_n$.

\subsection{Kronecker products of irreducible representations of $\CL(n)$}
\label{ss:clifford-Kp}
In this section we study the decomposition of the Kronecker products of irreducible representations
of $\CL(n)$ and of their restriction to the subgroup $\CL(n-1)$, yielding the proof of Theorem \ref{t:cln-wsr}.

\begin{proof}[Proof of Theorem \ref{t:cln-wsr}.(i)]
By Corollary \ref{c:mc1} for $H = G$, this is equivalent to
prove that the tensor product of any two irreducible representations of $\CL(n)$ decomposes multiplicitly free. 
To this end, we distinguish the two cases corresponding to the parity of $n$.

Suppose first that $n$ is even. By virtue of Lemma \ref{l:filippo} we only need to analyse the decomposition of
$\rho_n \otimes \rho_n$. Suppose $\theta \in \widehat{\CL(n)}$. Then we have
\[
\begin{split}
\langle \chi_{\rho_n \otimes \rho_n}, \chi_\theta \rangle & =
\frac{1}{2^{n+1}} \sum_{g \in \CL(n)} \chi_{\rho_n}(g)^2\overline{\chi_\theta(g)}\\ 
& = \frac{1}{2^{n+1}} \left[ \chi_{\rho_n}(1)^2\overline{\chi_\theta(1)}+  \chi_{\rho_n}(-1)^2\overline{\chi_\theta(-1)}\right]\\
& = \frac{1}{2^{n+1}} \left[ 2^n\overline{\chi_\theta(1)} +  2^n\overline{\chi_\theta(-1)}\right]\\
& = \begin{cases}
1 & \mbox{if  } \theta = \chi_A, \ A  \subseteq  X_n,  \ \mbox{by  \eqref{e:chiA}}\\
0  & \mbox{if  } \theta = \rho_n,  \  \  \   \  \  \   \  \  \   \  \  \   \  \mbox{by  \eqref{e:chirhon}}.
\end{cases}
\end{split}
\]
It follows that 
\begin{equation}
\label{e:n-n-A}
\rho_n\otimes \rho_n =\bigoplus_{A \subseteq X_n}\chi_A.
\end{equation}
Suppose now that $n$ is odd, say $n = 2m+1$.  As before, we need only to analyze $\rho_n^+\otimes \rho_n^+$,  $\rho_n^+\otimes \rho_n^-$  and  $\rho_n^-\otimes \rho_n^-$ and show that they decompose  multiplicity free.  Suppose $\theta \in \widehat{\CL(n)}$. Then we have
\[
\begin{split}
\langle \chi_{\rho_n \otimes \rho_n}, \chi_\theta \rangle & =
\frac{1}{2^{n+1}} \sum_{g \in \CL(n)} \chi_{\rho_n}(g)^2\overline{\chi_\theta(g)}\\ 
& = \frac{1}{2^{n+1}} \left[ \chi_{\rho_n}(1)^2\overline{\chi_\theta(1)}+  \chi_{\rho_n}(-1)^2\overline{\chi_\theta(-1)} + \right.\\
& \ \ \ \ \ \  \ \ \ \ \  +\left. \chi_{\rho_n}(\gamma_{X_n})^2\overline{\chi_\theta(\gamma_{X_n})}+  \chi_{\rho_n}(-\gamma_{X_n})^2\overline{\chi_\theta(-\gamma_{X_n})}\right]\\
& = \frac{1}{2^{n+1}} \left[ 2^{2m}(\overline{\chi_\theta(1)} +  \overline{\chi_\theta(-1)}) + c^22^{2m}(\overline{\chi_\theta(\gamma_{X_n})}+ \overline{\chi_\theta(-\gamma_{X_n})})\right]\\
& = \begin{cases}
1 & \mbox{if  } \theta = \chi_A,   \begin{array}{l}\mbox{with }|A| \mbox{ even and }  m\equiv 0  \mod 2\\
\mbox{or   } \ \ \  |A| \mbox{ odd and } \  m\equiv 1 \mod 2
\end{array}  \\
0  & \mbox{ otherwise }
\end{cases}
\end{split}
\]
where the last equality follows from \eqref{e:chirhon}, \eqref{e:trestelle} and \eqref{e:qua}.
It follows  that 
\begin{equation}
\label{e:tensore2}
\rho_n^+\otimes \rho_n^+ = \begin{cases}
\bigoplus_{|A| \mbox{ \small even}}\chi_A & \mbox{if } m \equiv 0 \mod 2\\
\bigoplus_{|A| \mbox{ \small odd}}\chi_A & \mbox{if } m \equiv 1 \mod 2.\\
\end{cases}
\end{equation}
Similarly we have 
\begin{equation}
\label{e:tensore3}
\rho_n^+\otimes \rho_n^- = \begin{cases}
\bigoplus_{|A| \mbox{ \small odd}}\chi_A & \mbox{if } m \equiv 0 \mod 2\\
\bigoplus_{|A| \mbox{ \small even}}\chi_A & \mbox{if } m \equiv 1 \mod 2,\\
\end{cases}
\end{equation}
while  
\begin{equation}
\label{e:tensore4}
\rho_n^-\otimes \rho_n^-  \sim \rho_n^+\otimes \rho_n^+.
\end{equation} 
\end{proof}

\begin{proof}[Proof of Theorem \ref{t:cln-wsr}.(ii)]
By Corollary \ref{c:mc1} this is equivalent to showing that the restriction to the subgroup $\CL(n-1)$
of the tensor product of any two irreducible representations of $\CL(n)$ decomposes multiplicitly free
if and only if $n$ is odd.

We first observe that 
\begin{equation}
\label{e:chi-diff-sim}
\chi_A \otimes \chi_B = \chi_{A \triangle B}
\end{equation}
and, if $n$ is even,
\begin{equation}
\label{e:tensor1}
\chi_A \otimes \rho_n = \rho_n
\end{equation}
while, if $n$ is odd,
\begin{equation}
\label{e:tensor-giusta}
\chi_A \otimes \rho_n^\pm = \begin{cases} \rho_n^\pm & \mbox{ if $|A|$ is even}\\
\rho_n^\mp & \mbox{ if $|A|$ is odd}
\end{cases}
\end{equation}
for all $A, B \subseteq X_n$.

Moreover,
\begin{equation}
\label{e:res-simple}
\Res^{\CL(n)}_{\CL(n-1)}\chi_A = \chi_{A \setminus \{n\}}
\end{equation}
for all $A \subseteq X_n$, and if $n$ is even,
\begin{equation}
\label{e:res-rho-piu}
\Res^{\CL(n)}_{\CL(n-1)}\rho_n = \rho_{n-1}^{+} \oplus \rho_{n-1}^{-},
\end{equation}
while, if $n$ is odd,
\begin{equation}
\label{e:res-rho}
\Res^{\CL(n)}_{\CL(n-1)}\rho_n^{\pm} = \rho_{n-1}.
\end{equation}

Collecting all these facts together, we deduce the following:
\begin{itemize}
\item $\Res^{\CL(n)}_{\CL(n-1)}\left(\chi_A \otimes \chi_B\right) = \chi_{(A \triangle B) \setminus \{n\}}$ (by
\eqref{e:chi-diff-sim} and \eqref{e:res-simple});
\item $\Res^{\CL(n)}_{\CL(n-1)}\left(\chi_A \otimes \rho_n\right) = \rho_{n-1}^{+} \oplus \rho_{n-1}^{-}$, if $n$ is even (by \eqref{e:tensor1} and \eqref{e:res-rho-piu});
\item $\Res^{\CL(n)}_{\CL(n-1)}\left(\chi_A \otimes \rho_n^\pm\right) = \rho_{n-1}$, if $n$ is odd (by \eqref{e:tensor-giusta} and \eqref{e:res-rho});
\item $\Res^{\CL(n)}_{\CL(n-1)}\left(\rho_{n} \otimes \rho_n\right) =  2 \bigoplus_{A \subseteq X_{n-1}} \chi_A$, if $n$ is even (by \eqref{e:n-n-A} and \eqref{e:res-simple});
\item $\Res^{\CL(n)}_{\CL(n-1)}\left(\rho_{n-1}^{\pm} \otimes \rho_{n-1}^{\pm}\right)= \Res^{\CL(n)}_{\CL(n-1)}\left(\rho_{n-1}^{\pm} \otimes \rho_{n-1}^{\mp}\right) =  \bigoplus_{A \subseteq X_{n-1}} \chi_A$, if $n$ is odd (by \eqref{e:tensore2}, \eqref{e:tensore3}, \eqref{e:tensore4} and \eqref{e:res-simple}).
\end{itemize}
The proof is now complete.
\end{proof}

\subsection{The orbits of $\CL(n)$ on $\CL(n) \times \CL(n)$}
\label{ss:clifford-orbits}
In this section we study the orbits of $\CL(n)$ acting by conjugation on the Cartesian product
$\CL(n) \times \CL(n)$. The interest for such analysis is motivated by the fact that spherical characters are
constant on each orbit.

Let $A,C \subseteq X$. Then we have
\begin{equation}
\label{e:18nov}
\begin{split}
\gamma_C^{-1} \gamma_A \gamma_C & = (-1)^{\frac{|C|(|C|-1)}{2} + \xi(A,C) + \xi(C, A \triangle C)} \gamma_{C \triangle(A \triangle C)}\\
& = (-1)^{\frac{|C|(|C|-1)}{2} + \xi(A,C) + \xi(C,A) + \xi(C,C)}\gamma_A\\
& =_{*} (-1)^{|A||C| - |A \cap C|}\gamma_A
\end{split}
\end{equation}
where $=_*$ follows from the fact that $\xi(C,C) = \frac{|C|(|C|-1)}{2}$. 

\begin{theorem}
Let $\emptyset \subseteq A,B \subseteq X_n$ and denote by $\mathcal{O}_n(\pm\gamma_A, \pm\gamma_B)$
the $\CL(n)$-orbit of $(\pm\gamma_A, \pm\gamma_B)$.  

\begin{enumerate}[{\rm (a)}]
\item{$\mathcal{O}_n(\pm\gamma_\emptyset, \pm\gamma_\emptyset) = \{(\pm\gamma_\emptyset, \pm\gamma_\emptyset)\}$;}
\item{if $A \neq B$ and $A, B \neq \emptyset, X_n$ then 
\begin{equation}
\label{e:orbita-1}
\mathcal{O}_n(\pm\gamma_A, \pm\gamma_B) = \{(\gamma_A, \gamma_B), (-\gamma_A, -\gamma_B), (-\gamma_A, \gamma_B), (\gamma_A, -\gamma_B)\}
\end{equation}
unless $n$ is odd and $A\coprod B = X_n$ (here and in the sequel $\coprod$ denotes a disjoint union) in which case \begin{equation}
\label{e:orbita-2}
\mathcal{O}_n(\pm\gamma_A, \pm\gamma_B) =\{(\pm\gamma_A, \pm\gamma_B), (\mp\gamma_A, \mp\gamma_B)\};
\end{equation}}
\item{if $A \neq \emptyset$ then 
\[
\mathcal{O}_n(\pm\gamma_\emptyset, \pm\gamma_A) = \{(\pm\gamma_\emptyset, \pm\gamma_A), (\pm\gamma_\emptyset, \mp\gamma_A)\}
\]
unless $n$ is odd and $A= X_n$, in which case $\mathcal{O}_n(\pm\gamma_\emptyset, \pm\gamma_{X_n}) = \{(\pm\gamma_\emptyset, \pm\gamma_{X_n})\}$. The analogous result holds for $\mathcal{O}_n(\pm\gamma_A, \pm\gamma_\emptyset)$;}
\item{if  $A \neq \emptyset$ then 
\[
\mathcal{O}_n(\pm\gamma_A, \pm\gamma_A) = \{(\pm\gamma_A, \pm\gamma_A), (\mp\gamma_A, \mp\gamma_A)\}
\]
unless $n$ is odd and $A= X_n$, in which case $\mathcal{O}_n(\pm\gamma_{X_n}, \pm\gamma_{X_n}) = \{(\pm\gamma_{X_n}, \pm\gamma_{X_n})\}$.}
\end{enumerate}
\end{theorem}
\begin{proof} We limit ourselves to prove (b) which is the most involved case: the remaining ones are simpler and can be treated in a similar manner.
Suppose first that $|A|$ and $|B|$ are both even. We distinguish two cases. (i) None of $A$ and $B$ is contained
in the other. Let then $a \in A \setminus B$ and $b \in B \setminus A$. Then taking $C = \{a,b\}$ (resp.
$\{a\}$, resp. $\{b\}$) and using \eqref{e:18nov} gives $\gamma_C(\gamma_A, \gamma_B)\gamma_C^{-1} =
(-\gamma_A, -\gamma_B)$ (resp. $(-\gamma_A, \gamma_B)$, resp. $(\gamma_A, -\gamma_B)$). (ii) Suppose
$A \\subseteq B$ (the specular case $B \\subseteq A$ is treated in the same way) and let $a \in A$ and $b \in B \setminus A$. Then taking $C = \{a\}$ (resp. $\{a,b\}$, resp. $\{b\}$) and using \eqref{e:18nov} gives $\gamma_C(\gamma_A, \gamma_B)\gamma_C^{-1} = (-\gamma_A, -\gamma_B)$ (resp. $(-\gamma_A, \gamma_B)$, resp. $(\gamma_A, -\gamma_B)$).
Thus \eqref{e:orbita-1} follows in both cases.
If $|A|$ and $|B|$ are both odd, similar calculations yield  again \eqref{e:orbita-1}.
We now consider the remaining case, namely when $|A|$ and $|B|$ have different parity. To fix ideas we suppose
that $|A|$ is even and $|B|$ is odd. We distinguish three cases.
(i) $A \cup B \neq X_n$ and $A \setminus B \neq \varnothing$. Let then $a \in A \setminus B$ and $c \in X_n \setminus (A \cup B)$. Then taking $C = \{a\}$ (resp. $\{a,c\}$, resp. $\{c\}$) and using \eqref{e:18nov} gives $\gamma_C(\gamma_A, \gamma_B)\gamma_C^{-1} = (-\gamma_A, -\gamma_B)$ (resp. $(-\gamma_A, \gamma_B)$, resp. $(\gamma_A, -\gamma_B)$). (ii) $A \\subseteq B$. Let then $a \in A$, $b \in B \setminus A$ and $c \in X_n \setminus B$.
Then taking $C = \{a,c\}$ (resp. $\{a\}$, resp. $\{b,c\}$) and using \eqref{e:18nov} gives $\gamma_C(\gamma_A, \gamma_B)\gamma_C^{-1} = (-\gamma_A, -\gamma_B)$ (resp. $(-\gamma_A, \gamma_B)$, resp. $(\gamma_A, -\gamma_B)$). 
(iii) $A \cup B = X_n$. Suppose that $A \cap B \neq \varnothing$. Let then $a \in A \setminus B$ and $d \in A \cap B$ and set $C = \{a\}$ (resp. $\{d\}$, resp. $\{a,d\}$). We have $\gamma_C(\gamma_A, \gamma_B)\gamma_C^{-1} = (-\gamma_A, -\gamma_B)$ (resp. $(-\gamma_A, \gamma_B)$, resp. $(\gamma_A, -\gamma_B)$) and \eqref{e:orbita-1} follows also in this case. On the other hand, if $A \cap B = \varnothing$ (so that necessarily $n = |A| + |B|$ is odd) and $C \subseteq X_n$, then if $|C \cap A|$ is even (resp. odd) then $\gamma_C(\gamma_A, \gamma_B)\gamma_C^{-1} = (\gamma_A, \gamma_B)$
(resp. $\gamma_C(\gamma_A, \gamma_B)\gamma_C^{-1} = (-\gamma_A, -\gamma_B)$) and \eqref{e:orbita-2} follows.
\end{proof}

\subsection{The spherical characters for $(\CL(n) \times\CL(n) \times\CL(n), \widetilde{\CL(n)})$}
\label{ss:clifford-sc}
In this section we study the spherical characters associated with the Gelfand pair $(\CL(n) \times\CL(n) \times\CL(n), \widetilde{\CL(n)})$. We start with a simple preliminary combinatorial result. This could be immediately deduced from
the orthogonality relations for the characters, but we prefer to give a direct proof.
\begin{lemma}
\label{l:piccolo}
Let $U \subseteq X_n$. Then
\begin{equation}
\label{e:piccola}
\frac{1}{2^n} \sum_{D \subseteq X_n} (-1)^{\vert U \cap D\vert} =
\begin{cases} 1 & \mbox{ if } U = \varnothing\\
0 &   \mbox{ otherwise.}
\end{cases}
\end{equation}
\end{lemma}
\begin{proof} If $U = \varnothing$, we have $\vert U \cap D\vert = \varnothing$ for every $D \subseteq X_n$ so that
each summand in the Left Hand Side of \eqref{e:piccola} equals $1$ and therefore the whole Left Hand Side of \eqref{e:piccola} equals $1$. Suppose now that $U\neq\varnothing$ and fix $u \in U$. Consider the sets
$\PP_0 = \{A \subseteq X_n: \vert A \cap U \vert \mbox{ is even}\}$ and $\PP_1 = \{A \subseteq X_n: \vert A \cap U \vert \mbox{ is odd}\}$. Since the map $\Phi \colon \PP_0 \to \PP_1$ defined by
\[
\Phi(A) = \begin{cases} A \setminus \{u\} & \mbox{ if } u \notin A\\
A \cup \{u\} & \mbox{ otherwise}
\end{cases}
\]
is bijective, we have $\vert \PP_0 \vert = \vert \PP_1 \vert$. As a consequence,
\[
\frac{1}{2^n} \sum_{D \subseteq X_n} (-1)^{\vert U \cap D\vert} = \frac{1}{2^n} \left(\sum_{A \in \PP_0} (-1)^{\vert U \cap D\vert} + \sum_{A \in \PP_1} (-1)^{\vert U \cap D\vert}\right) = \frac{1}{2^n} \left(\vert \PP_0 \vert - \vert \PP_1 \vert\right) = 0.
\]
\end{proof}

The spherical character on $\CL(n) \times\CL(n) \times\CL(n)$ associated with the representation $\chi_A \boxtimes \chi_B \boxtimes \chi_C$, for $A,B,C \subseteq X_n$, is given by
\begin{equation}
\label{e:dpiccola}
\psi_{\chi_A \boxtimes \chi_B \boxtimes \chi_C}(\varepsilon_1\gamma_{T_1},\varepsilon_2\gamma_{T_2},\varepsilon_3\gamma_{T_3}) = \frac{1}{2^n} \sum_{D \subseteq X_n} (-1)^{\vert A \cap(D \triangle T_1)\vert+ \vert B \cap(D \triangle T_2)\vert+ \vert C \cap(D \triangle T_3)\vert}
\end{equation}
for all $\varepsilon_i = \pm 1$ and $T_i \subseteq X_n$, $i=1,2,3$.
In particular (when $T_1 = T_2 = T_3 = \varnothing$), we have
\begin{equation}
\label{e:buona}
\psi_{\chi_A \boxtimes \chi_B \boxtimes \chi_C}(1,1,1) = \begin{cases} 1 & \mbox{ if } C = A \triangle B\\
0 & \mbox{ otherwise.}
\end{cases}
\end{equation}
Indeed, in this case, \eqref{e:dpiccola} gives
\[
\begin{split}
\psi_{\chi_A \boxtimes \chi_B \boxtimes \chi_C}(1,1,1) & = \frac{1}{2^n} \sum_{D \subseteq X_n} (-1)^{\vert A \cap D\vert + \vert B \cap D\vert + \vert C \cap D\vert}\\
& = \frac{1}{2^n} \sum_{D \subseteq X_n} (-1)^{\vert (A \triangle B) \cap D\vert + 2 \vert A \cap B \cap D \vert + \vert C \cap D\vert}\\
& = \frac{1}{2^n} \sum_{D \subseteq X_n} (-1)^{\vert ((A \triangle B) \triangle C) \cap D\vert + 2\vert ((A \triangle B) \cap C) \cap D\vert}\\
& = \frac{1}{2^n} \sum_{D \subseteq X_n} (-1)^{\vert ((A \triangle B) \triangle C) \cap D\vert}
\end{split}
\]
and \eqref{e:buona} follows from Lemma \ref{l:piccolo} by taking $U = (A \triangle B) \triangle C$.

Suppose now that $n$ is even. 
The spherical character on $\CL(n) \times\CL(n) \times\CL(n)$ associated with the representation $\rho_n \boxtimes \rho_n \boxtimes \rho_n$ is trivial.
Indeed,
\[
\begin{split}
\psi_{\rho_n \boxtimes \rho_n \boxtimes \rho_n}(\varepsilon_1\gamma_{T_1}&,\varepsilon_2\gamma_{T_2} ,\varepsilon_3\gamma_{T_3}) = 
\frac{1}{2^{n+1}} \sum_{\varepsilon = \pm 1} \sum_{D \subseteq X_n} \prod_{i=1}^3\chi_{\rho_n}(\varepsilon \varepsilon_i \gamma_D \gamma_{T_i})\\
& = \begin{cases}
\frac{1}{2^{n+1}} \sum\nolimits_{\varepsilon = \pm 1} \prod_{i=1}^3\chi_{\rho_n}(\varepsilon \varepsilon_i (-1)^{\xi(T_i,T_i)}\gamma_\varnothing) & \mbox{ if } T_1 = T_2 = T_3\\
0 & \mbox{ otherwise}
\end{cases}\\
& = \begin{cases}
\sum_{\varepsilon = \pm 1} \varepsilon \left(\frac{1}{2^{n+1}} \prod_{i=1}^3 (\varepsilon_i (-1)^{\xi(T_i,T_i)} 2^{n/2})\right) & \mbox{ if } T_1 = T_2 = T_3\\
0 & \mbox{ otherwise}
\end{cases}\\
& = 0
\end{split}
\]
for all $\varepsilon_i = \pm 1$ and $T_i \subseteq X_n$, $i=1,2,3$.

The spherical character on $\CL(n) \times\CL(n) \times\CL(n)$ associated with the representation $\chi_A \boxtimes \rho_n \boxtimes \rho_n$ for $A \subseteq X_n$, is given by 
\[
\begin{split}
\psi_{\chi_A \boxtimes \rho_n \boxtimes \rho_n}(\varepsilon_1\gamma_{T_1}&,\varepsilon_2\gamma_{T_2},\varepsilon_3\gamma_{T_3})  = 
\frac{1}{2^{n+1}} \sum_{\varepsilon = \pm 1} \sum_{D \subseteq X_n} \chi_A(\varepsilon \varepsilon_1 \gamma_D \gamma_{T_1}) \prod_{i=2}^3\chi_{\rho_n}(\varepsilon \varepsilon_i \gamma_D \gamma_{T_i})\\
& = \begin{cases}
\frac{1}{2^{n+1}} \sum_{\varepsilon = \pm 1} (-1)^{|T \cap T_1|} \prod_{i=2}^3(\varepsilon \varepsilon_i (-1)^{\xi(T,T)}2^{n/2}) & \mbox{ if } T_2 = T_3 =:T\\
0 & \mbox{ otherwise}
\end{cases}\\
& = \begin{cases}
(-1)^{|T \cap T_1|}\varepsilon_2\varepsilon_3 & \mbox{ if } T_2 = T_3 =:T\\
0 & \mbox{ otherwise}
\end{cases}
\end{split}
\]
for all $\varepsilon_i = \pm 1$ and $T_i \subseteq X_n$, $i=1,2,3$.

The spherical character on $\CL(n) \times\CL(n) \times\CL(n)$ associated with the representation $\chi_A \boxtimes \chi_B \boxtimes \chi_{\rho_n}$, for $A,B \subseteq X_n$, is again trivial.
Indeed,
\[
\begin{split}
\psi_{\chi_A \boxtimes \chi_B \boxtimes \rho_n}(\varepsilon_1\gamma_{T_1},\varepsilon_2\gamma_{T_2},\varepsilon_3\gamma_{T_3}) & = \frac{1}{2^{n+1}} \sum_{\varepsilon = \pm 1} \sum_{D \subseteq X_n} \chi_A(\varepsilon \varepsilon_1 \gamma_D\gamma_{T_1}) 
\chi_B(\varepsilon \varepsilon_2 \gamma_D\gamma_{T_2})\chi_{\rho_n}(\varepsilon \varepsilon_3 \gamma_D \gamma_{T_3})\\ 
& = \frac{1}{2^{n+1}} \sum_{\varepsilon = \pm 1} \varepsilon \left(\chi_A(\gamma_{T_3}\gamma_{T_1}) 
\chi_B(\gamma_{T_3}\gamma_{T_2})(\varepsilon_3 (-1)^{\xi(T_3,T_3)}2^{n/2})\right)\\ 
& = 0
\end{split}
\]
for all $\varepsilon_i = \pm 1$ and $T_i \subseteq X_n$, $i=1,2,3$.

Suppose now that $n$ is odd. 
The spherical character on $\CL(n) \times\CL(n) \times\CL(n)$ associated with the representations $\rho_n^{\eta_1} \boxtimes \rho_n^{\eta_2} \boxtimes \rho_n^{\eta_3}$ is trivial for all $\eta_i = \pm$, $i=1,2,3$.
Indeed,
\[
\begin{split}
\psi_{\rho_n^{\eta_1} \boxtimes \rho_n^{\eta_2} \boxtimes \rho_n^{\eta_3}}(\varepsilon_1\gamma_{T_1},\varepsilon_2\gamma_{T_2},\varepsilon_3\gamma_{T_3}) & = 
\frac{1}{2^{n+1}} \sum_{\varepsilon = \pm 1} \sum_{D \subseteq X_n} \prod_{i=1}^3\chi_{\rho_n^{\eta_i}}(\varepsilon \varepsilon_i \gamma_D \gamma_{T_i})\\
& = \frac{1}{2^{n+1}} \sum_{\varepsilon = \pm 1} \varepsilon \left(\sum_{D \subseteq X_n}\prod_{i=1}^3\chi_{\rho_n^{\eta_1}}(\varepsilon_i \gamma_D \gamma_{T_i})\right)\\
& = 0
\end{split}
\]
for all $\varepsilon_i = \pm 1$ and $T_i \subseteq X_n$, $i=1,2,3$.

In order to express the spherical character on $\CL(n) \times\CL(n) \times\CL(n)$ associated with the representations $\chi_A \boxtimes \rho_n^{\eta_2} \boxtimes \rho_n^{\eta_3}$ for $A \subseteq X_n$ and $\eta_i = \pm$, $i=2,3$, let
$\varepsilon_i = \pm 1$ and $T_i \subseteq X_n$, $i=1,2,3$. Then we have
\[
\begin{split}
\psi_{\chi_A \boxtimes \rho_n^{\eta_2} \boxtimes \rho_n^{\eta_3}} (\varepsilon_1\gamma_{T_1},\varepsilon_2\gamma_{T_2},\varepsilon_3\gamma_{T_3}) & =
\frac{1}{2^{n+1}} \sum_{\varepsilon = \pm 1} \sum_{D \subseteq X_n} \chi_A(\varepsilon \varepsilon_1 \gamma_D \gamma_{T_1})\prod_{i=2}^3\chi_{\rho_n^{\eta_i}}(\varepsilon \varepsilon_i \gamma_D \gamma_{T_i})\\
& = \frac{1}{2^{n}}\sum_{D \subseteq X_n} \chi_A(\gamma_{D \triangle T_1})\prod_{i=2}^3 \varepsilon_i\chi_{\rho_n^{\eta_i}}((-1)^{\xi(D,T_i)}\gamma_{D \triangle T_i}).
\end{split}
\]
%\begin{comment}
This quantity vanishes if $T_2 \neq T_3, \overline{T_3}$ while,
if $T_2 = T_3 := T$ equals
\[
\begin{split}
\frac{1}{2^{n}}(\chi_A(\gamma_{T \triangle T_1})  &  \varepsilon_2\varepsilon_3 \chi_{\rho_n^{\eta_2}}
((-1)^{\xi(T,T)}\gamma_{T \triangle T})\chi_{\rho_n^{\eta_3}}((-1)^{\xi(T,T)}\gamma_{T \triangle T})\\ 
& + \chi_A(\gamma_{\overline{T} \triangle T_1}) \varepsilon_2\varepsilon_3 \chi_{\rho_n^{\eta_2}}((-1)^{\xi(\overline{T},T)}\gamma_{\overline{T} \triangle T})\chi_{\rho_n^{\eta_3}}((-1)^{\xi(\overline{T},T)}\gamma_{\overline{T} \triangle T}))\\
& = 
2^{n}\varepsilon_2\varepsilon_3\left((-1)^{\vert A \cap (T \triangle T_1)\vert}  
+ c^2\eta_2\eta_3(-1)^{\vert A \cap (\overline{T} \triangle T_1)\vert}\right) 
\end{split}
\]
%\end{comment}
and, finally, if $T_2 = \overline{T_3} := T$, equals
%\begin{comment}
\[
\begin{split}
\frac{1}{2^{n}}(\chi_A(\gamma_{T \triangle T_1}) & \varepsilon_2\varepsilon_3 \chi_{\rho_n^{\eta_2}}((-1)^{\xi(T,T)}\gamma_{T \triangle T})\chi_{\rho_n^{\eta_3}}((-1)^{\xi(T,\overline{T})}\gamma_{T \triangle \overline{T}}) \\ 
&  +  \chi_A(\gamma_{\overline{T} \triangle T_1}) \varepsilon_2\varepsilon_3 \chi_{\rho_n^{\eta_2}}((-1)^{\xi(\overline{T},T)}\gamma_{\overline{T} \triangle T})\chi_{\rho_n^{\eta_3}}((-1)^{\xi(\overline{T},\overline{T})}\gamma_{\overline{T} \triangle \overline{T}}))\\
& = 
2^{n}c\varepsilon_2\varepsilon_3\left(\eta_3(-1)^{\vert A \cap (T \triangle T_1)\vert + \xi(T,T)+ \xi(T,\overline{T})}  
+ \eta_2(-1)^{\vert A \cap (\overline{T} \triangle T_1)\vert + \xi(T,\overline{T})+ \xi(\overline{T},\overline{T})}\right).  
\end{split}
\]

As special cases (essentially the only ones for which we have a rather simple expression) we have
\[
\psi_{\chi_A \boxtimes \rho_n^{\eta_2} \boxtimes \rho_n^{\eta_3}} (\varepsilon_1\gamma_{T},\varepsilon_2\gamma_{T},\varepsilon_3\gamma_{T}) =
2^n \varepsilon_2\varepsilon_3\left(1+(-1)^{\vert A \vert}\eta_2\eta_3 c^2\right)
\]
and 
\[
\psi_{\chi_A \boxtimes \rho_n^{\eta_2} \boxtimes \rho_n^{\eta_3}} (\varepsilon_1\gamma_{\overline{T}},\varepsilon_2\gamma_{T},\varepsilon_3\gamma_{T}) =
2^n \varepsilon_2\varepsilon_3(-1)^{\vert A \vert}
\]

%\end{comment}
Finally, the spherical character on $\CL(n) \times\CL(n) \times\CL(n)$ associated with the representations $\chi_A \boxtimes \chi_B \boxtimes \rho_n^{\pm}$ is again trivial for all $A,B \subseteq X_n$.
Indeed,
\[
\begin{split}
\psi_{\chi_A \boxtimes \chi_B \boxtimes \rho_n^\pm}(\varepsilon_1\gamma_{T_1},\varepsilon_2\gamma_{T_2},\varepsilon_3\gamma_{T_3}) & = 
\frac{1}{2^{n+1}} \sum_{\varepsilon = \pm 1} \sum_{D \subseteq X_n} \chi_A(\varepsilon \varepsilon_1 \gamma_D \gamma_{T_1}) \chi_B(\varepsilon \varepsilon_2 \gamma_D \gamma_{T_2}) \chi_{\rho_n^\pm}(\varepsilon \varepsilon_3 \gamma_D \gamma_{T_3})\\
& = \frac{1}{2^{n+1}} \sum_{\varepsilon = \pm 1} \varepsilon \left(\sum_{D \subseteq X_n}
\chi_A(\gamma_D \gamma_{T_1}) \chi_B(\gamma_D \gamma_{T_2})
\chi_{\rho_n^\pm}(\varepsilon_3 \gamma_D \gamma_{T_3})\right)\\
& = 0
\end{split}
\]
for all $\varepsilon_i = \pm 1$ and $T_i \subseteq X_n$, $i=1,2,3$.

%%%%%%%%%%%%%%%%%%%%%%%%%%%%%%%%%%%%%%%%%%%%%%%%%%%%%
%%%%%%%%%%%%%%%%%%%%%%%%%%%%%%%%%%%%%%%%%%%%%%%%%%%%%


\begin{thebibliography}{99}
%\bibitem{tree} T. Ceccherini-Silberstein, F. Scarabotti and F. Tolli,  Trees, wreath products and finite Gelfand pairs, {\it Adv. Math.}, {\bf 206} (2006), no. 2, 503--537.
\bibitem{GelGeorg} T. Ceccherini-Silberstein,  F. Scarabotti and F. Tolli, Finite Gelfand pairs and their applications to probability and statistics, {\it  J. Math. Sci. (N. Y.)} {\bf 141} (2007), no. 2, 1182--1229.
\bibitem{book} T. Ceccherini-Silberstein, F. Scarabotti and F. Tolli, {\it Harmonic analysis on finite groups:
representation theory, Gelfand pairs and Markov chains.}  
Cambridge Studies in Advanced Mathematics {108}, Cambridge University Press 2008.
%\bibitem{Mackey} T. Ceccherini-Silberstein, A. Mach\`i, F. Scarabotti, F. Tolli,  Induced representations and Mackey theory, Functional analysis, {\it J. Math. Sci. (N. Y.)} {\bf 156} (2009), no. 1, 11--28.
%\bibitem{Clifford} T. Ceccherini-Silberstein, F. Scarabotti and F. Tolli, Clifford theory and applications, Functional analysis, {\it J. Math. Sci. (N. Y.)} {\bf 156} (2009), no. 1, 29--43.
%\bibitem{wreath} T. Ceccherini-Silberstein, F. Scarabotti and F. Tolli, Representation theory of wreath products of finite groups,  Functional analysis, {\it J. Math. Sci. (N. Y.)} {\bf 156} (2009), no. 1, 44--55.
%\bibitem{book2} T. Ceccherini-Silberstein, F. Scarabotti and F. Tolli, {\it Representation theory of the symmetric groups: the Okounkov-Vershik approach, character formulas, and partition algebras.} Cambridge Studies in Advanced Mathematics {121}, Cambridge University Press 2010.
\bibitem{book3} T. Ceccherini-Silberstein, F.Scarabotti and F.Tolli:  {\it Representation Theory and Harmonic Analysis of wreath products of finite groups}. London Mathematical Society Lecture Note Series {\bf 410},  Cambridge University Press, 2014.
\bibitem{AM} T. Ceccherini-Silberstein, F. Scarabotti and F. Tolli, Mackey's theory of $\tau$-conjugate representations for finite groups,  arXiv:1311.7252.
%\bibitem{DD1} D. D'Angeli and A. Donno, Crested products of Markov chains, {\it  Ann. Appl. Probab.}, {\bf  19}, no. 1, (2009), 414--453.
% \bibitem{DD2} D. D'Angeli and A. Donno, Markov chains on orthogonal block structures, {\it  European J. Combin.}, {\bf  31}, Issue 1 (2010), 34--46.
%\bibitem{DD3} D. D'Angeli and A. Donno, Generalized crested products, {\it  European J. Combin.}, {\bf  32}, Issue 2 (2011), 243--257.
\bibitem{Diaconis} P. Diaconis, {\it Groups Representations in Probability and Statistics.} IMS Hayward, CA, 1988.
% \bibitem{Macdonald} I.~G. Macdonald, {\it Symmetric functions and Hall Polynomials}, second edition. Oxford University Press, 1995.
%\bibitem{MIZ000} H. Mizukawa, Zonal spherical functions on the complex reflection groups and (n+1,m+1)-hypergeometric functions, {\it  Adv. Math.} {\bf 184} (2004) 1--17.
% \bibitem{Mizukawa} H. Mizukawa, {Twisted Gelfand pairs of complex reflection groups and $ r$-congruence properties of Schur functions}, {\it Ann. Comb.} {\bf 15} (2011), no. 1,109--125.
% \bibitem{NS} M.A. Naimark and A.I. Stern, {\it Theory of Group Representations.} Springer-Verlag, New York, 1982.
% \bibitem{ScarabottiLapBer} F. Scarabotti, Time to reach stationarity in the Bernoulli-Laplace diffusion model with many urns, {\it Adv. in Appl. Math.} {\bf 18} (1997), no. 3, 351--371.
\bibitem{st1} F. Scarabotti and F. Tolli, Harmonic analysis on a finite homogeneous space, {\it Proc. Lond. Math. Soc.}(3) {\bf{100}} (2010),  no. 2, 348--376.
%\bibitem{st2} F. Scarabotti and F. Tolli, Harmonic analysis on a finite homogeneous space II: the Gelfand Tsetlin decomposition,  {\it Forum Mathematicum}  {\bf{22}} (2010), 897--911.
%\bibitem{st3}  F. Scarabotti and F. Tolli, Fourier analysis of subgroup-conjugacy invariant functions on finite groups, {\it Monatsh. Math.} {\bf 170} (2013) 465–-479.
%\bibitem{st4}  F. Scarabotti and F. Tolli, Hecke algebras and harmonic analysis on  finite groups,  {\it Rend. Mat. Appl.} (7) {\bf 33} (2013), no. 1-2, 27–-51.
\bibitem{Simon} B. Simon, {\it Representations of finite and compact groups}, American Math. Soc., 1996.
% \bibitem{Sternberg} S. Sternberg, {\it Group theory and physics.} Cambridge University Press, Cambridge, 1994.
%\bibitem{Terras} A. Terras, {\it Fourier analysis on finite groups and applications}. London Mathematical Society Student Texts, 43. Cambridge University Press, Cambridge, 1999.
\end{thebibliography}
\end{document}